\documentclass[12pt]{article}
\usepackage{mysty}
\usepackage{amsthm,amsmath,amssymb}
\usepackage{rotating}

\makeatletter
  \def\my@tag@font{\normalsize}
  \def\maketag@@@#1{\hbox{\m@th\normalfont\my@tag@font#1}}
  \let\amsmath@eqref\eqref
  \renewcommand\eqref[1]{{\let\my@tag@font\relax\amsmath@eqref{#1}}}
\makeatother

\usepackage[colorlinks=true,citecolor=black,linkcolor=black,urlcolor=blue]{hyperref}
\newcommand{\doi}[1]{\href{http://dx.doi.org/#1}{\texttt{doi:#1}}}
\newcommand{\arxiv}[1]{\href{http://arxiv.org/abs/#1}{\texttt{arXiv:#1}}}

\newcommand{\abs}[1]{\lvert#1\rvert}
\newcommand{\sign}[1]{\textnormal{sign}(#1)}
\newcommand{\mylog}[1]{\textit{l}_{#1}}
\renewcommand{\H}[2]{\textnormal{H}_{#1}(#2)}
\newcommand{\R}{\mathbb R}
\newcommand{\Z}{\mathbb Z}
\newcommand{\N}{\mathbb N}
\newcommand{\K}{{\mathbb K}}
\renewcommand{\S}[2]{\textnormal{S}_{#1}\left(#2\right)}
\newcommand{\eg}{e.g.,\ }
\newcommand{\ie}{i.e.,\ }

\newcommand{\HarmonicSumsP}{\texttt{HarmonicSums}}

\newcommand{\GL}[1]{\textnormal{G}\left(#1\right)}

\theoremstyle{plain}
\newtheorem{theorem}{Theorem}
\newtheorem{lemma}[theorem]{Lemma}

\theoremstyle{definition}

\newtheorem{example}[theorem]{Example}

\theoremstyle{remark}

\allowdisplaybreaks[4]

\title{\bf Discovering and Proving Infinite\\ Pochhammer Sum Identities}

\author{Jakob Ablinger\thanks{Supported by the Austrian Science Fund (FWF) grant SFB F50 (F5009-N15), by the strategic program ``Innovatives O\"O 2020'' by the Upper Austrian Government and by the bilateral project DNTS-Austria 01/3/2017 (WTZ BG
03/2017), funded by Bulgarian National Science Fund and OeAD (Austria).}\\
\small Research Institute for Symbolic Computation\\[-0.8ex]
\small Johannes Kepler University\\[-0.8ex] 
\small Linz, Austria\\
\small\tt jablinge@risc.jku.at
}

\date{
\dateline{February 28, 2019}{}
}

\begin{document}

\maketitle

\begin{abstract}
We consider nested sums involving the Pochhammer symbol at infinity and rewrite them in terms of a small set of constants, such as powers of $\pi,$ $\log(2)$ or zeta values. 
In order to perform these simplifications, we view the series as specializations of generating series. 
For these generating series, we derive integral representations in terms of root-valued iterated integrals or directly in terms of cyclotomic harmonic polylogarithms. 
Using substitutions, we express the root-valued iterated integrals as cyclotomic 
harmonic polylogarithms. Finally, by applying known relations among the cyclotomic harmonic polylogarithms, we derive expressions in terms of several constants. 
We provide an algorithimc machinery to prove identities which so far could only be proved using classical hypergeometric approaches.
These methods are implemented in the computer algebra package \HarmonicSumsP.
\end{abstract}


\section{Infinite Nested Pochhammer Sums}
\label{sec:1}
The goal of this article is to find and prove identities of the following form: 
\begin{eqnarray*}
\sum\limits_{n=1}^{\infty }  \frac{\left(\frac{1}{2}\right)_n \sum\limits_{i=1}^n \frac{\sum\limits_{j=1}^i \frac{1}{j^2}}{i}}{(n+1)!}&=&3 \zeta _3,\\
\sum\limits_{n=1}^{\infty }  \frac{\left(\frac{1}{2}\right)_n \sum\limits_{i=1}^n \frac{\sum\limits_{j=1}^i \frac{1}{j}}{i^3}}{(n+1)!}&=&\frac{2 l_2^4}{3}-4 \zeta _3 l_2+\frac{2}{3} \pi ^2 l_2^2+16 p_4-\frac{13 \pi ^4}{180},\\
\sum\limits_{n=1}^{\infty } \frac{\left(\frac{1}{4}\right)_n \sum\limits_{i=1}^n \frac{1}{i^2}}{(n+1)!}&=&\frac{7 \pi ^2}{18}-\frac{16 C}{3}-6 \mylog{2}^2+2 \pi  \mylog{2},\\
\sum\limits_{n=1}^{\infty }  \frac{\left(\frac{1}{3}\right)_n \left(\sum\limits_{i=1}^n \frac{1}{3i+1}-\sum\limits_{i=1}^n \frac{1}{3i+2}\right)}{(n+1)!}&=&\frac{\pi }{\sqrt{3}}-\frac{3}{4}-\frac{\sqrt{3 \pi } \Gamma \left(\frac{5}{6}\right)}{\sqrt[3]{2} \Gamma \left(\frac{1}{3}\right)},\\
\sum _{n=1}^{\infty } \frac{\left(\frac{1}{2}\right)_n \sum\limits_{i=1}^n \frac{\sum\limits_{j=1}^i \frac{1}{2 j+1}}{2 i+1}}{(2n+1)^2 n!}&=&\frac{1}{96} \pi  \left(4 \pi ^2 \mylog{2}-72 \mylog{2}^2+56 \mylog{2}^3-9 \zeta_{3}\right),\\
\sum\limits_{i=1}^{\infty } \frac{ \left(\frac{1}{3}\right)_n\sum\limits_{i=1}^n \frac{\sum\limits_{j=1}^{i} \frac{\sum\limits_{k=1}^{j} \frac{\sum\limits_{l=1}^{k} \frac{\sum\limits_{m=1}^{l} \frac{1}{m}}{l}}{k}}{j}}{i}}{(n+1)!}&=&180 \zeta_{5}-\frac{\pi ^5}{\sqrt{3}},\\
\sum _{n=1}^{\infty } \frac{\left(\frac{1}{2}\right)_n \sum _{i=1}^n \frac{1}{i^9}}{(n+1)!}&=&-\frac{2339 \pi ^8 \mylog{2}}{453600}-\frac{79 \pi ^6 \mylog{2}^3}{5670}-\frac{1}{75} \pi ^4 \mylog{2}^5-\frac{8}{945} \pi ^2 \mylog{2}^7+\frac{8 \mylog{2}^9}{2835}\\
    &&-\frac{79 \pi ^6 \zeta _{3}}{3780}-\frac{1}{5} \pi ^4 \mylog{2}^2 \zeta _{3}-\frac{4}{9} \pi ^2 \mylog{2}^4 \zeta _{3}+\frac{16}{45} \mylog{2}^6 \zeta _{3}-\frac{4}{3} \pi ^2 \mylog{2} \zeta _{3}^2\\
    &&+\frac{16}{3} \mylog{2}^3 \zeta _{3}^2+\frac{8 \zeta _{3}^3}{3}-\frac{3 \pi ^4 \zeta _{5}}{10}-4 \pi ^2 \mylog{2}^2 \zeta_{5}+8 \mylog{2}^4 \zeta _{5}+48 \mylog{2} \zeta _{3} \zeta _{5}\\
    &&-6 \pi ^2 \zeta _{7}+72 \mylog{2}^2 \zeta _{7}+\frac{340 \zeta _{9}}{3},\\
\end{eqnarray*}
where $(x)_n$ denotes the Pochhammer symbol, $\mylog{k}:=\log(k),\ \zeta_k:=\sum_{n=1}^\infty\frac{1}{n^k}$ and $C$ denotes Catalan's constant.

Note that similar identities were given in~\cite{Liu:2019,BinomialSumIdentities}. Such identities are of interest in physics: in particular, such sums have been studied in order to perform calculations of higher 
order corrections to scattering processes in particle physics \cite{Ablinger:2013eba,Davydychev:2001,Davydychev:2003mv,Fleischer:1998nb,Jegerlehner:2003,Kalmykov:2000qe,Kalmykov:2007dk,Ogreid:1997bx,Weinzierl:2004bn}. 
Moreover similar identities were also considered in \cite{Borwein:2001,Borwein:2000,Lehmer:1985,ZhiWei:2014,Zucker:1985}, 
and there is a connection to Ap\'ery's proof of the irrationality of $\zeta(3)$ (see \cite{Borwein:1987}).
\cite{Weinzierl:2004bn,ZhiWei:2014,Zucker:1985}.

While \cite{BinomialSumIdentities} basically deals with sums of the form 
$$\sum _{n=1}^{\infty } \frac{\left(\frac{1}{2}\right)_n}{n!}f(n) \textnormal{ and } \sum _{n=1}^{\infty } \frac{\left(-\frac{1}{2}\right)_n}{n!}f(n),$$
we are going to consider a much wider class of sums in the frame of this paper. In addition we will state a general computer algebra method to evaluate a large class of sums in terms of nested integrals. 
Moreover, we will be able to prove a structural theorem, about when such a sum can be expressed in terms of so called \textit{cyclotomic polylogarithms}.

The main purpose of this article is to present methods which can be automated, hence not all identities presented in this paper are new identities. 
To make more precise which class of sums we are considering, some definitions are in place. Let $r\in\N$ and let $a_i,c_i\in\N$ and $b_i\in\N_0$ for $1\leq i\leq r$ then we call $\S{(a_1,b_1,c_1),\ldots,(a_r,b_r,c_r)}n$ defined as
\begin{eqnarray}\label{cyclosum}
  \S{(a_1,b_1,c_1),\ldots,(a_r,b_r,c_r)}n:=\sum_{i_1=1}^n\frac{1}{(a_1 i_1+b_1)^{c_1}}\sum_{i_2=1}^{i_1}\frac{1}{(a_2 i_2+b_2)^{c_2}}\cdots\sum_{i_r=1}^{i_{r-1}}\frac{1}{(a_r i_r+b_r)^{c_r}}
\end{eqnarray}
a \textit{cyclotomic harmonic sum} (compare~\cite{Ablinger:2013jta,Ablinger:2013hcp,Ablinger:2013eba,Ablinger:2011te}) of depth $r$. 
Note that if $a_i=1$ and $b_i=0$ for $1\leq i\leq r$ we write
\begin{eqnarray}\label{hsum}
\S{c_1,c_2,\ldots,c_r}n:=\S{(1,0,c_1),(1,0,c_2),\ldots,(1,0,c_r)}n,
\end{eqnarray}
and we call $\S{c_1,c_2,\ldots,c_r}n$ a 
\textit{multiple harmonic sum} (see, e.g.,~\cite{Ablinger:2013hcp,Ablinger:2013cf,Bluemlein1999,Bluemlein2000,Vermaseren1998}).

The sums we are considering take the form
\begin{eqnarray}\label{Pochhammersum}
  \sum_{n=1}^\infty \frac{(p)_n}{(an+b)^c(n+d)!}f(n),
\end{eqnarray}
where $a,b,c,d\in\N_0,\ p\in\R$ and $f(n)$ is a cyclotomic harmonic sum. We will refer to~(\ref{Pochhammersum}) as \textit{Pochhammer sum}.

We are going to find representations of these Pochhammer sums in terms of special classes of integrals (that are similar to the iterated integrals in~\cite{Ablinger:2014} and correspond to the iterated integrals in~\cite{BinomialSumIdentities}). 
These classes of integrals are iterated integrals over \textit{hyperexponential} functions. More precisely a function $f(x)$ is called \textit{hyperexponential} if
$$\frac{f^\prime(x)}{f(x)}=q(x),$$
where $q(x)$ is a rational function in $x.$

Then an \textit{iterated integral} over the hyperexponential functions $f_1(x),f_2(x),\ldots,f_k(x)$ is defined recursively by
$$
\GL{f_1(\tau),f_2(\tau),\cdots,f_k(\tau),x}=\int_0^xf_1(\tau_1)\GL{f_2(\tau),\cdots,f_k(\tau),\tau_1}d\tau_1,
$$
with the special case $\GL{x}=1.$
Since some letters might have a non-integrable singularity at the base point $x=0$ we consistently define
$$
\GL{f(\tau),x}:=\int_0^x \left(f(t)-\frac{c}{t}\right)dt+c\log(x),
$$
where $c$ takes the unique value such that the integrand on the right hand side is integrable at $t=0.$ It is important to note that this definition preserves the derivative $\frac{d}{dx}\GL{f(\tau),x}~=~f(x).$ In general, we set
\begin{eqnarray*}
\GL{f_1(\tau),\ldots,f_j(\tau),x}&:=&\int_0^x \left(f_1(t)\GL{f_2(\tau),\ldots,f_j(\tau),t}-\sum_{i=0}^kc_i\frac{\log(t)^i}{t}\right)dt\\&&+\sum_{i=0}^k \frac{c_i}{i+1}\log(x)^{i+1},
\end{eqnarray*}
where $k$ and $c_0,\ldots,c_k$ are chosen to remove any non-integrable singularity. Again the result is unique and retains $$\frac{d}{dx}\GL{f_1(\tau),\ldots,f_j(\tau),x}=f_1(x)\GL{f_2(\tau),\ldots,f_j(\tau),x}.$$

In the following we will define a subclass of iterated integrals (compare~\cite{Ablinger:2011te}). For $a \in \N$ and $b \in \N,$ $b < \varphi(a),$ where $\varphi$ denotes Euler's totient function, we define 
\begin{eqnarray}
&&f_a^b:(0,1)\mapsto \R\nonumber\\
&&f_a^b(x)=\left\{ 
		\begin{array}{ll}
				\frac{1}{x}, &  \textnormal{if }a=b=0  \\
				\frac{x^b}{\Phi_a(x)}, & \textnormal{otherwise},
		\end{array} 
		\right.  \nonumber
\end{eqnarray}
where $\Phi_a(x)$ denotes the $a$th cyclotomic polynomial, \eg the first cyclotomic polynomials are given by
\begin{eqnarray*}
\Phi_1(x) &=& x - 1 \\
\Phi_2(x) &=& x + 1 \\
\Phi_3(x) &=& x^2 + x + 1 \\
\Phi_4(x) &=& x^2 + 1 \\
\Phi_5(x) &=& x^4 + x^3 + x^2 + x+ 1~~\text{etc.}
\end{eqnarray*}
Now, let $m_i=(a_i,b_i) \in \N^2,$ $b_i<\varphi(a_i);$ for $x\in (0,1)$ we define \textit{cyclotomic polylogarithms} recursively as follows (compare \eg~\cite{Ablinger:2011te}):
\begin{eqnarray}
\H{}{x}&=&1,\nonumber\\
\H{m_1,\ldots,m_k}{x} &=&\left\{ 
		  	\begin{array}{ll}
						\frac{1}{k!}(\log{x})^k,&  \textnormal{if }m_i=(0,0)\\
						  &\\
						\int_0^x{f_{a_1}^{b_1}(y) \H{m_2,\ldots,m_k}{y}dy},& \textnormal{otherwise}. 
			\end{array} \right.  \nonumber
\end{eqnarray}
We call $k$ the weight of a cyclotomic polylogarithm and in case the limit exists we extend the definition to $x=1$ and write
$$
\H{m_1,\ldots,m_k}{1}:=\lim_{x\to1}\H{m_1,\ldots,m_k}{x}.
$$
Note that restricting the alphabet to the letters $(0,0),(1,0)$ and $(2,0)$ leads to \textit{harmonic polylogarithms}~\cite{Remiddi:1999ew}.

The proposed strategy to prove and find Pochhammer sum identities reads as follows and follows the method proposed in~\cite{BinomialSumIdentities}):
\begin{description}\label{genmethod}
 \item[Step 1:] Rewrite the sums in terms of nested integrals.
 \item[Step 2:] Rewrite the integrals in terms of cyclotomic polylogarithms (see~\cite[Section~4]{BinomialSumIdentities}).
 \item[Step 3:] Provide a sufficiently strong database to eliminate relations among these cyclotomic polylogarithms and find reduced expressions (see Section~4).
\end{description}
This article focuses on Step~1 and we will present three different possibilities to find integral representations of Pochhammer sums. In order to accomplish this task, we view infinite sums as 
specializations of generating functions~\cite{BinomialSumIdentities,Ablinger:2014}. Namely, if we are given an
integral representation of the generating function of a sequence, then we can obtain an integral representation for the
infinite sum over that sequence if the limit $x \to 1$ can be carried out. This approach to infinite sums can be summarized by the
following formula:
\[
 \sum_{i=1}^\infty f(i)  = \lim_{x\to1}\sum_{i=1}^\infty x^if(i).
\] 
For details on Step~2 (implemented in the command \texttt{SpecialGLToH} in \HarmonicSumsP) and on Step~3 we refer to~\cite{BinomialSumIdentities}. It has to be mentioned that we computed and used relation tables of harmonic polylogarithms at one up to weight~12, 
for cyclotomic polylogarithms of cyclotomy~4 and~6 we computed and used relation tables of cyclotomic polylogarithms at~1 up to weight~6. The size of these tables amounts to several gigabytes.
Note that the full strategy has been implemented in the Mathematica package {\tt HarmonicSums}\footnote{The package {\tt HarmonicSums} can be downloaded at\\ \url{http://www.risc.jku.at/research/combinat/software/HarmonicSums}.}\cite{HarmonicSums}.

To complete this introduction we define a number of constants that will appear throughout this article:

\begin{tabular}{lll}
$\mylog{2}:=\log(2)$ & $\mylog{3}:=\log(3)$ & $\zeta_3:=S_{3}(\infty );$\\
$\zeta_5:=S_{5}(\infty );$ & $\zeta_7:=S_{7}(\infty );$ &$\zeta_9:=S_{9}(\infty );$\\
$\zeta_{11}:=S_{11}(\infty );$&$C:= \text{Catalan};$ & $p_ 4:= \text{Li}_ 4\left(\frac{1}{2}\right);$\\
$p_ 5:= \text{Li}_ 5\left(\frac{1}{2}\right);$ & $p_ 6:= \text{Li}_ 6\left(\frac{1}{2}\right);$ & $p_ 7:= \text{Li}_ 7\left(\frac{1}{2}\right);$\\
$p_ 8:= \text{Li}_ 8\left(\frac{1}{2}\right);$ & $p_ 9:= \text{Li}_ 9\left(\frac{1}{2}\right);$ & $s_1:=S_{-5,-1}(\infty );$\\
$s_2:=S_{5,-1,-1}(\infty );$ & $s_3:=S_{-5,1,1}(\infty );$ & $s_4:=S_{5,3}(\infty );$\\
$s_5:=S_{-7,-1}(\infty );$ & $s_6:=S_{-5,-1,-1,-1}(\infty );$ & $s_7:=S_{-5,-1,1,1}(\infty );$\\
$h_1:=H_{(3,0),(0,0)}(1);$ & $h_2:=H_{(3,0),(0,0),(1,0)}(1); $ & $h_3:=H_{(3,0),(0,0),(0,0),(0,0)}(1); $\\
$h_4:=H_{(3,0),(0,0),(1,0),(1,0)}(1); $ & $h_6:=H_{(5,1)}(1); $ & $h_6:=H_{(5,3)}(1); $\\
$h_7:=H_{(5,1),(0,0)}(1); $ & $h_8:=H_{(5,2),(0,0)}(1); $
\end{tabular}

Here we extend the definition~(\ref{hsum}) to negative indices by
$$\S{c_1,c_2,\ldots,c_r}n:=\sum_{i_1=1}^n\frac{\sign{c_1}^{i_1}}{\abs{i_1}^{c_1}}\sum_{i_2=1}^{i_1}\frac{\sign{c_2}^{i_2}}{\abs{i_2}^{c_2}}\cdots\sum_{i_r=1}^{i_{r-1}}\frac{\sign{c_r1}^{i_r}}{\abs{i_r}^{c_r}}.$$
Note that these constants do not possess any further relations induced by the algebraic properties given in~\cite[Section~4]{BinomialSumIdentities}, namely shuffle, stuffle, multiple argument and duality relations.

In the following sections we will use different methods to compute integral representations of the generating function. In Section~\ref{sec:2} we will use holonomic closure properties while in Section~\ref{sec:3} and~\ref{sec:4} we will use rewrite rules.
In Section~\ref{sec:4} we will consider a subclass of Pochhammer sums, for which we can directly find representations in terms of cyclotomic polylogarithms \ie we do not have to deal with Step~2 of the proposed strategy.

\section{Using Closure Properties of Holonomic Functions to derive Generating Functions}
\label{sec:2}
In the following we repeat important definitions and properties (compare~\cite{Ablinger:2014,InvMellin,KauersPaule:2011}).
Let $\mathbb K$ be a field of characteristic~0. A function $f=f(x)$ is called \textit{holonomic} (or \textit{D-finite}) if there exist 
polynomials $p_d(x),p_{d-1}(x),\ldots,p_0(x)\in \mathbb K[x]$  (not all $p_i$
being $0$) such that the following holonomic differential equation holds:
\begin{equation}
 p_d(x)f^{(d)}(x)+\cdots+p_1(x)f'(x)+p_0(x)f(x)=0.
\end{equation}
We emphasize that the class of holonomic functions is rather large due to its
closure properties. Namely, if we are given two such differential
equations that contain holonomic functions $f(x)$ and $g(x)$ as solutions, one
can compute holonomic differential equations that contain $f(x)+g(x)$,
$f(x)g(x)$ or $\int_0^x f(y)dy$ as solutions. In other words any composition
of these operations over known holonomic functions $f(x)$ and $g(x)$ is again a
holonomic function $h(x)$. In particular, if for the inner building blocks
$f(x)$ and $g(x)$ the holonomic differential equations are given, also the holonomic
differential equation of $h(x)$ can be computed.\\
Of special importance is the connection to recurrences. A sequence $(f_n)_{n\geq0}$ with $f_n\in\mathbb K$ is called
holonomic (or \textit{P-finite}) if there exist polynomials 
$p_d(n),p_{d-1}(n),\ldots,p_0(n)\in \mathbb K[n]$ (not all $p_i$ being $0$) such
that the holonomic recurrence
\begin{equation}
 p_d(n)f_{n+d}+\cdots+p_1(n)f_{n+1}+p_0(n)f_n=0
\end{equation}
holds for all $n\in\mathbb N$ (from a certain point on).
In the following we utilize the fact that holonomic functions are
precisely the generating functions of holonomic sequences: 
if $f(x)$ is holonomic, then the coefficients 
$f_n$ of the formal power series expansion 
$$f(x) = \sum\limits_{n=0}^{\infty} f_n x^n$$
form a holonomic sequence. Conversely, for a given holonomic sequence
$(f_n)_{n\geq0}$, the function defined by the above sum (\ie its 
generating function) is holonomic (this is true in the sense of formal power series, even if the sum has a zero radius of 
convergence). Note that given a holonomic differential equation for a holonomic function $f(x)$ it is straightforward to 
construct a 
holonomic recurrence for the coefficients of its power series expansion. For a
recent overview of this holonomic machinery and further literature we
refer to~\cite{KauersPaule:2011}.

Since cyclotomic sums are holonomic sequences with respect to~$n$ and the iterated integrals we consider are holonomic functions with respect to~$x,$ we can use holonomic closure properties to derive integral representations of Pochhammer sums:
Given a Pochhammer sum 
$$\sum_{n=1}^\infty \frac{(p)_n}{(an+b)^c(n+d)!}g(n),$$ where $g(n)$ is a cyclotomic sum. We proceed as proposed in on page~\pageref{genmethod}: define $$f_n:=\frac{(p)_n}{(an+b)^c(n+d)!}g(n)$$ and try to find an iterated integral representation of 
$$f(x):=\sum_{n=1}^\infty x^n f_n$$ using the following steps:
\begin{enumerate}
 \item Compute a holonomic recurrence equation for $(f_n)_{n\geq0}.$
 \item Compute a holonomic differential for $f(x).$
 \item Compute initial values for the differential equation.
 \item Solve the differential equation to get a closed form representation for $f(x).$
\end{enumerate}
This procedure is implemented in the packages \HarmonicSumsP\ and can be called by
$$
\textbf{ComputeGeneratingFunction}\left[\frac{(p)_n}{(an+b)^c(n+d)!}g(n),x,\{n,1,\infty \}\right].
$$
We will succeed in finding a closed form representation for $f(x)$ in terms of iterated integrals, if we can find a full solution set of the derived differential equation. 
The command~\texttt{ComputeGeneratingFunction} internally uses the differential solver implemented in~\HarmonicSumsP, which finds all solutions of holonomic differential equations that can be expressed in
terms of iterated integrals over hyperexponential alphabets~\cite{InvMellin,Ablinger:2014,Bronstein,Singer:99,Petkov:92}; these solutions
are called d'Alembertian solutions~\cite{Abramov:94}, in addition for differential equations of order two it finds all solutions that are Liouvillian~\cite{InvMellinKovacic,Kovacic,Singer:99}.

If we succeed in finding a closed form representation for $f(x)$ in terms of iterated integrals, we proceed with Step~2 and Step~3 of the proposed strategy. 
Hence we send $x\to 1$ and try to transform these iterated integrals to expression in terms of cyclotomic polylogarithms and finally we use 
relations between cyclotomic polylogarithms at one to derive an expression in terms of known constants.

The Pochhammer sum
\begin{eqnarray}\label{GeneralExampleSum}
\sum _{n=1}^{\infty } \frac{\left(-\frac{1}{2}\right)_nS_1(n) }{(3+n)^2 (n-1)!}
\end{eqnarray}
will deal as a representative example to illustrate all three different methods that are presented in this article. First, we work out the sum using the method presented above.
\begin{example}\label{GeneralExample}
We consider the sum~(\ref{GeneralExampleSum}) and start to derive a recurrence for $$f_n:=\frac{\left(-\frac{1}{2}\right)_nS_1(n) }{(3+n)^2 (n-1)!};$$ we find:
\begin{eqnarray*}
&&(2+n) (4+n)^2 (1+2 n) (3+2 n) f_{n}-2 (1+n) (5+n)^2 (3+2 n) (5+2 n) f_{n+1}\\&&+4 (1+n) (2+n) (3+n) (6+n)^2 f_{n+2}=0.
\end{eqnarray*}
Using the closure properties of holonomic functions we find the following differential equation 
\begin{eqnarray*}
&&96 f(x)+3 (-250+343 x) f'(x)+3 \left(144-590 x+481 x^2\right) f''(x)\\&&+x \left(352-942 x+599 x^2\right) f^{(3)}(x)+8 x^2 \left(9-20 x+11 x^2\right) f^{(4)}(x)\\&&+4 (-1+x)^2 x^3 f^{(5)}(x)=0,
\end{eqnarray*}
satisfied by $$\sum _{n=1}^{\infty } x^n\frac{\left(-\frac{1}{2}\right)_nS_1(n) }{(3+n)^2 (n-1)!}.$$
We can solve this differential equation for example using the differential equation solver implemented in \HarmonicSumsP:
\begin{eqnarray*}
\textbf{SolveDE}\bigl[&&\hspace{-0.7cm}96 f[x]+3 (-250+343 x) f'[x]+3 \left(144-590 x+481 x^2\right) f''[x]\\&&\hspace{-0.7cm}+x \left(352-942 x+599 x^2\right) f^{(3)}[x]+8 x^2 \left(9-20 x+11 x^2\right) f^{(4)}[x]\\
&&\hspace{-0.7cm}+4 (-1+x)^2 x^3 f^{(5)}[x]==0,f[x],x\bigr].
\end{eqnarray*}
By checking initial values we find
\small
\begin{eqnarray}
&&\frac{1}{7350 x^3}\Biggl(
-1776+808 x-319 x^2-888 x^3-600 x^4+6656 \biggl[\text{G}\left(\frac{\sqrt{1-\tau }}{\tau };x\right)-\text{G}\left(\frac{1}{\tau };x\right)\biggr]\nonumber\\
&&+3360\Biggl[ \text{G}\left(\frac{1}{\tau },\frac{1}{\tau };x\right)-\text{G}\left(\frac{1}{\tau },\frac{\sqrt{1-\tau }}{\tau };x\right)
+\text{G}\left(\frac{\sqrt{1-\tau }}{\tau },\frac{1}{1-\tau };x\right)+ \text{G}\left(\frac{\sqrt{1-\tau }}{\tau },\frac{1}{\tau };x\right)\nonumber\\
&&- \text{G}\left(\frac{\sqrt{1-\tau }}{\tau },\frac{\sqrt{1-\tau }}{\tau };x\right)\biggr]
+4\sqrt{1-x}\biggl(\left(404-218 x-111 x^2-75 x^3\right) \biggl[\text{G}\left(\frac{1}{1-\tau };x\right)\nonumber\\
&&+ \text{G}\left(\frac{1}{\tau };x\right)- \text{G}\left(\frac{\sqrt{1-\tau }}{\tau };x\right)\biggr]+222 \left(2-x-x^2\right)\biggr)
\Biggr).
\label{ExampleGLRep}
\end{eqnarray}
\normalsize
At this point we send $x\to 1$ and use the command \texttt{SpecialGLToH} in \HarmonicSumsP\ to derive an expression in terms of cyclotomic polylogarithms (compare~\cite[Section 3]{BinomialSumIdentities}). This leads to
\begin{eqnarray*}
&&-\frac{9367}{7350}-\frac{3328 H_{(0,0)}(1)}{3675}+\frac{8}{35} H_{(0,0)}(1){}^2-\frac{64 H_{(2,0)}(1)}{3675}-\frac{32}{35} H_{(2,0)}(1){}^2-\frac{16}{35} H_{(0,0),(0,0)}(1)\\&&-\frac{32}{35} H_{(0,0),(1,0)}(1)-\frac{32}{35}
   H_{(2,0),(0,0)}(1)+\frac{64}{35} H_{(2,0),(1,0)}(1)+\frac{64}{35} H_{(2,0),(2,0)}(1).
\end{eqnarray*}
Finally, we can use relations between cyclotomic polylogarithms at one (compare~\cite[Section 4]{BinomialSumIdentities}) to derive
\begin{eqnarray}\label{ExampleResult}
 \sum _{n=1}^{\infty } \frac{\left(-\frac{1}{2}\right)_nS_1(n) }{(3+n)^2 (n-1)!}=\frac{-9367+560 \pi ^2+6720 \mylog{2}^2-128 \mylog{2}}{7350}.
\end{eqnarray}
Note that in the last step of this example we are actually only dealing with harmonic polylogarithms (see~\cite{Remiddi:1999ew}).
\end{example}

Let us now list several identities that could be computed using this method:
\begin{eqnarray*}
\sum _{n=1}^{\infty } \frac{\left(\frac{1}{3}\right)_n S_{1,1,1}(n)}{(n+1)!}&=&18 \zeta _3-\frac{\pi ^3}{\sqrt{3}},\\
\sum _{n=1}^{\infty } \frac{\left(\frac{1}{3}\right)_n S_2(n)}{(n+1)!}&=&\frac{5 \pi ^2}{16}+\frac{27 h_1}{8}+\frac{3}{8} \sqrt{3} \pi  l_3-\frac{27 l_3^2}{16},\\
\sum\limits_{n=1}^{\infty } \frac{\left(\frac{1}{4}\right)_n \S{2}{n}}{(n+1)!}&=&-\frac{16 C}{3}+\frac{7 \pi ^2}{18}-6 \mylog{2}^2+2 \pi  \mylog{2},\\
\sum _{n=1}^{\infty } \frac{\left(\frac{1}{2}\right)_n S_{(2,1,1)}(n)}{(2 n+1)^2 n!}&=&\frac{1}{4} \pi  \mylog{2} (3 \mylog{2}-2),\\
\sum _{n=1}^{\infty } \frac{\left(\frac{1}{2}\right)_n S_{(2,1,1),(2,1,1)}(n)}{(2n+1)^2 n!}&=&\frac{1}{96} \pi  \left(4 \pi ^2 \mylog{2}-72 \mylog{2}^2+56 \mylog{2}^3-9 \zeta_{3}\right).
\end{eqnarray*}
Several formulas that can be found in~\cite{Liu:2019} can be also discovered and proved using the described method. Here we are going to list some of them:
\begin{eqnarray*}
\sum _{n=0}^{\infty } \frac{\left(\frac{1}{2}\right)_n \left(S_1(n){}^2-S_2(n)\right)}{(n+1)!}&=&8 l_2^2+\frac{2 \pi ^2}{3},\\
\sum _{n=0}^{\infty } \frac{\left(\frac{1}{2}\right)_n \left(S_1(n){}^3-3 S_1(n) S_2(n)+2 S_3(n)\right)}{(n+1)!}&=&24 \zeta _3+16 l_2^3+4 \pi ^2 l_2,\\
\sum _{n=0}^{\infty } \frac{\left(\frac{1}{4}\right)_n \left(S_1(n){}^3-3 S_1(n) S_2(n)+2 S_3(n)\right)}{(n+1)!}&=&-96 C l_2+16 \pi  C+72 \zeta _3+36 l_2^3-18 \pi  l_2^2\\&&+13 \pi ^2 l_2-\frac{9 \pi ^3}{2},\\
\sum _{n=0}^{\infty } \frac{\left(\frac{1}{4}\right)_n \left(S_1(n){}^2-S_2(n)\right)}{(n+1)!}&=&288 C l_2+48 \pi  C+216 \zeta _3+108 l_2^3+54 \pi  l_2^2\\&&+39 \pi ^2 l_2+\frac{27 \pi ^3}{2},\\
\sum _{n=0}^{\infty } \frac{\left(\frac{1}{4}\right)_n \left(S_1(n){}^2-S_2(n)\right)}{(n+1)!}&=&-\frac{32 C}{3}+12 l_2^2-4 \pi  l_2+\frac{13 \pi ^2}{9},\\
\sum _{n=0}^{\infty } \frac{\left(\frac{3}{4}\right)_n \left(S_1(n){}^2-S_2(n)\right)}{(n+1)!}&=&32 C+36 l_2^2+12 \pi  l_2+\frac{13 \pi ^2}{3},\\
\sum _{n=1}^{\infty } \frac{\left(\frac{1}{2}\right)_{n-1} \left(S_1(n){}^2-2 S_1(n)+S_2(n)\right)}{n!}&=&8.
\end{eqnarray*}

Note that this method can also be used to compute integral representations of sums of the form
\begin{eqnarray*}
\sum _{n=1}^{\infty } \frac{x^n (3)_n S_{3}(n)}{n^2 n!}.
\end{eqnarray*}
Here we find 
\begin{eqnarray*}
\sum _{n=1}^{\infty } \frac{x^n (3)_n S_{3}(n)}{n^2 n!}&=&H_{(0,0),(1,0),(0,0),(0,0),(1,0)}(x)-\frac{3 \text{Li}_2(x){}^2}{4}-\frac{\text{Li}_3(x)}{2 (-1+x)}\\&&-\frac{3}{2} \log (1-x) \text{Li}_3(x)+\frac{3 \text{Li}_4(x)}{2}+\text{Li}_5(x).
\end{eqnarray*}
and sending for instance $x\to \frac{1}2$ we get:
\begin{eqnarray*}
\sum _{n=1}^{\infty } \frac{\left(\frac{1}{2}\right)^n (3)_n S_{3}(n)}{n^2 n!}.&=&-\frac{\pi^4}{192}-\frac{\pi^2l_2}{12}+\frac{\pi^4l_2}{288}-\frac{1}{16}\pi^2l_2^2+\frac{l_2^3}{6}-\frac{5}{72}\pi^2l_2^3+\frac{l_2^4}{16}+\frac{11l_2^5}{120}\\
&&+\frac{3p_4}{2}+3l_2p_4+4p_5+\frac{7\zeta_3}{8}-\frac{7\pi^2\zeta_3}{48}+\frac{21l_2\zeta_3}{16}+\frac{7}{8}l_2^2\zeta_3-\frac{81\zeta_5}{64}.
\end{eqnarray*}
Finally, we consider 
\begin{eqnarray}\label{S11Example}
 \sum _{n=1}^{\infty } \frac{\left(\frac{1}{2}\right)_n S_{11}(n)}{(n+1)!},
\end{eqnarray}
proceeding as proposed, we find a differential equation of order 16:
\small
\begin{eqnarray*}
&&430080 f(x)+210 (-4096+1592275 x) f'(x)\\
&&+42 \left(-33554432-6356812 x+407269601 x^2\right) f''(x)\\
&&+\left(671088640-33746963856 x-8037305736 x^2+192200072453 x^3\right) f^{(3)}(x)\\
&&+x \left(11047661360-204994450032 x-61653276602 x^2+771941124781 x^3\right) f^{(4)}(x)\\
&&+13 x^2 \left(3812823056-38280317036 x-13991732902 x^2+109483643797 x^3\right) f^{(5)}(x)\\
&&+26 x^3 \left(3555308396-22952549314 x-9866689087 x^2+53652573053 x^3\right) f^{(6)}(x)\\
&&+572 x^4\left(152216474-697858881 x-344085550 x^2+1394066246 x^3\right) f^{(7)}(x)\\
&&+143 x^5 \left(323583896-1123610312 x-623388464 x^2+1975831409 x^3\right) f^{(8)}(x)\\
&&+143 x^6 \left(103854560-285705072 x-175728306 x^2+451597351 x^3\right) f^{(9)}(x)\\
&&+286 x^7 \left(10492016-23636810 x-15928139 x^2+34105982 x^3\right) f^{(10)}(x)\\
&&+3 x^8 \left(130094536-246156812 x-180008400 x^2+328091581 x^3\right) f^{(11)}(x)\\
&&+x^9 \left(32842216-53242200 x-41920782 x^2+66163633 x^3\right) f^{(12)}(x)\\
&&+x^{10} \left(1762640-2487876 x-2095294 x^2+2904173 x^3\right) f^{(13)}(x)\\
&&+2 x^{11} \left(28900-35986 x-32239 x^2+39703 x^3\right) f^{(14)}(x)\\
&&+4 x^{12} \left(262-291 x-276 x^2+305 x^3\right) f^{(15)}(x)\\
&&+8 (-1+x)^2 x^{13} (1+x) f^{(16)}(x)=0.
\end{eqnarray*}
\normalsize
Solving this differential equation is possible but takes quite some time, so this indicates, that we might look for more feasible methods to find generating function representations for Pochhammer sums of that kind. 
In the following sections we will introduce rewrite rules, which will allow to compute generating function representations 
of Pochhammer sums without having to solve differential equations.
\section{Using Rewrite Rules to derive Generating Functions}
\label{sec:3}
In this section we are going to state rewrite rules which will allow us to find integral representations of the generating functions of Pochhammer sums without having to solve differential equations. 
We will summarize these rewrite rules in the following lemmas. We start with the base cases where there is no inner sum present:

\begin{lemma}\label{GLBaseCase} Let $\K$ be a field of characteristic 0. Then the following identities hold in the ring $\K[[x]]$ of formal power series with $a,c\in\N$ and $b,d\in\Z$:
 \begin{eqnarray}
   \sum_{n=1}^\infty x^n\frac{(p)_n}{(n+d)!} &=& \frac{(1-x)^{d-p}}{x^d}(p)_{-d},\ d<0, \label{eq:GenFunPochhammerSum01}\\
   \sum_{n=1}^\infty x^n\frac{(p)_n}{n} &=& (1-x)^{-p}-1,\label{eq:GenFunPochhammerSum02}\\
   \sum_{n=1}^\infty x^n\frac{(p)_n}{(n+d)!} &=& (1-x)^{d-p}\frac{p}{d!x^d}\int_0^x (1-t)^{p-d-1}t^d dt,\ d>0, \label{eq:GenFunPochhammerSum03}\\
   \sum_{n=1}^\infty x^n\frac{(p)_n}{(a\,n+b)^c(n+d)!} &=& \frac{x^{-\frac{b}{a}}}{a}\int_0^xt^{\frac{b}{a}-1}\sum_{n=1}^\infty t^n \frac{(p)_n}{(a\,n+b)^{c-1}(n+d)!}dt.\label{eq:GenFunPochhammerSum04}
 \end{eqnarray}
\end{lemma}
In case an inner sum is present we will make use of the following three lemmas.
\begin{lemma}\label{GLc0dl0} Let $\K$ be a field of characteristic 0 and let $f:\N \to \K.$ Then the following identity holds in the ring $\K[[x]]$ of formal power series with $d<0$:
 \begin{eqnarray}
   &&\sum_{n=1}^\infty x^n\frac{(p)_n}{(n+d)!}\sum_{i=1}^nf(i) = \label{eq:GenFunPochhammerSum2}\\
     &&\hspace{1cm}=\frac{(1-x)^{d-p}}{x^d}\left((p)_{-d}\sum_{i=1}^{-d}f(i)+\int_0^x\frac{(1-t)^{p-d-1}}{t^{1-d}}\sum_{n=1}^\infty t^n \frac{(p)_n}{(n+d-1)!}f(n)dt\right).\nonumber
 \end{eqnarray}
\end{lemma}
\begin{proof}
 Both sides satisfy the following initial value problem for $y(x),$ which has a unique solution near $x=0:$
 \begin{eqnarray*}
 y'(x)-\frac{px-d}{(1-x)x}y(x)&=&\frac{1}{(1-x)x}\sum_{n=1}^{\infty}x^n\frac{(p)_n}{(n+d-1)!}f(n),\\
 y(0)&=&0.
 \end{eqnarray*}
\end{proof}

\begin{lemma}\label{GLc0dgeq0} Let $\K$ be a field of characteristic 0 and let $f:\N \to \K.$ Then the following identity holds in the ring $\K[[x]]$ of formal power series with $d\geq0$:
 \begin{eqnarray}
   &&\sum_{n=1}^\infty x^n\frac{(p)_n}{(n+d)!}\sum_{i=1}^nf(i) =\label{eq:GenFunPochhammerSum1}\\
   &&\hspace{2cm}=\frac{(1-x)^{d-p}}{x^d}\int_0^xt^{d-1}(1-t)^{p-d-1}\sum_{n=1}^\infty t^n \frac{(p)_n}{(n+d-1)!}f(n)dt.\nonumber
 \end{eqnarray}
\end{lemma}
\begin{proof}
 Both sides satisfy the following initial value problem for $y(x),$ which has a unique solution near $x=0:$
 \begin{eqnarray*}
 y'(x)-\frac{p\,x-d}{(1-x)x}y(x)&=&\frac{1}{(1-x)x}\sum_{n=1}^{\infty}x^n\frac{(p)_n}{(n+d-1)!}f(n),\\
 y(0)&=&0.
 \end{eqnarray*}
\end{proof}

\begin{lemma}\label{GLcgeq1} Let $\K$ be a field of characteristic 0 and let $f:\N \to \K.$ Then the following identity holds in the ring $\K[[x]]$ of formal power series with $a,c\in\N$ and $b\in\Z$:
 \begin{eqnarray}
   &&\sum_{n=1}^\infty x^n\frac{(p)_n}{(a\,n+b)^c(n+d)!}\sum_{i=1}^nf(i)=\label{eq:GenFunPochhammerSum3}\\
   &&\hspace{2cm}=\frac{x^{-\frac{b}{a}}}{a}\int_0^xt^{\frac{b}{a}-1}\sum_{n=1}^\infty t^n \frac{(p)_n}{(a\,n+b)^{c-1}(n+d)!}\sum_{i=1}^nf(i)dt.\nonumber
 \end{eqnarray}
\end{lemma}
\begin{proof}
 Both sides satisfy the following initial value problem for $y(x),$ which has a unique solution near $x=0:$
 \begin{eqnarray*}
 y'(x)-\frac{b}{a\,x}y(x)&=&\frac{1}{a\,x}\sum_{n=1}^{\infty} x^n\frac{(p)_n}{(a\,n+b)^{c-1}(n+d)!}\sum_{i=1}^nf(i),\\
 y(0)&=&0.
 \end{eqnarray*}
\end{proof}

Note that formulas related to the previous lemmas concerning binomial sums can be found in~\cite{Ablinger:2014}.

Let us now, for the second time, consider~(\ref{GeneralExampleSum}) and illustrate how the previous lemmas can be used as rewrite rules to find integral representations of Pochhammer sums.
\begin{example}
We again look for a closed form representation in terms of iterated integrals of 
$$\sum _{n=1}^{\infty } x^n \frac{\left(-\frac{1}{2}\right)_nS_1(n) }{(3+n)^2 (n-1)!}.$$
We start by using Lemma~\ref{GLcgeq1} twice:
\begin{eqnarray*}
  \sum _{n=1}^{\infty } x^n\frac{\left(-\frac{1}{2}\right)_nS_1(n) }{(3+n)^2 (n-1)!}&=&x^{-3}\int_0^xt^2 \sum _{n=1}^{\infty } t^n\frac{S_1(n) \left(-\frac{1}{2}\right)_n}{(3+n) (n-1)!}dt\\
  &=&x^{-3}\int_0^xt^{-1} \int_0^tu^2 \sum _{n=1}^{\infty } u^n\frac{S_1(n) \left(-\frac{1}{2}\right)_n}{(n-1)!}dudt.
\end{eqnarray*}
Now we apply Lemma~\ref{GLc0dl0} followed by applying~(\ref{eq:GenFunPochhammerSum04}) and~(\ref{eq:GenFunPochhammerSum01})
\begin{eqnarray*}
  &&\sum _{n=1}^{\infty } x^n\frac{\left(-\frac{1}{2}\right)_nS_1(n) }{(3+n)^2 (n-1)!}\\
  &&\hspace{1cm}=x^{-3}\int_0^xt^{-1} \int_0^t\frac{u^3}{\sqrt{1-u}}\biggl(\int_0^u\frac{1}{v^2\sqrt{1-v}}\sum _{n=1}^{\infty } v^n\frac{\left(-\frac{1}{2}\right)_n}{n(n-2)!} dv-\frac{1}{2}\biggr)dudt\\
  &&\hspace{1cm}=x^{-3}\int_0^xt^{-1} \int_0^t\frac{u^3}{\sqrt{1-u}}\biggl(\int_0^u\frac{1}{v^2\sqrt{1-v}}\int_0^v\frac{1}{w}\sum_{n=1}^\infty w^n \frac{(-\frac{1}{2})_n}{(n-2)!}dw dv-\frac{1}{2}\biggr)dudt\\
  &&\hspace{1cm}=x^{-3}\int_0^xt^{-1} \int_0^t\frac{u^3}{\sqrt{1-u}}\biggl(\int_0^u\frac{1}{v^2\sqrt{1-v}}\int_0^v \frac{-w}{4(1-w)^{\frac{3}{2}}} dw dv-\frac{1}{2}\biggr)dudt.
\end{eqnarray*}
At this point we rewrite the expression in terms of iterated integrals (this can be done by hand or by using the command \texttt{GLIntegrate} of \HarmonicSumsP) and arrive again at~(\ref{ExampleGLRep}) 
and hence we can proceed as in Example~\ref{GeneralExample} to arrive at 
\begin{eqnarray*}
 \sum _{n=1}^{\infty } \frac{\left(-\frac{1}{2}\right)_nS_1(n) }{(3+n)^2 (n-1)!}=\frac{-9367+560 \pi ^2+6720 \mylog{2}^2-128 \mylog{2}}{7350}.
\end{eqnarray*}
\end{example}

Note that this method is implemented in the package \HarmonicSumsP\ using the command \texttt{PochhammerSumToGL}. Calling
\begin{eqnarray*}
\textbf{PochhammerSumToGL}\left[\frac{\left(-\frac{1}{2}\right)_n S_1(n)}{(3+n)^2 (-1+n)!},x,\{n,1,\infty \}\right]
\end{eqnarray*}
will immediately give~(\ref{ExampleGLRep}).

Reconsidering~(\ref{S11Example}) we find 
\begin{eqnarray*}
 \sum _{n=1}^{\infty } \frac{\left(\frac{1}{2}\right)_n S_{11}(n)}{(n+1)!}&=&
  -4 \text{G}\left(\frac{1}{\tau },\frac{1}{\tau },\frac{1}{\tau },\frac{1}{\tau},\frac{1}{\tau },\frac{1}{\tau },\frac{1}{\tau },\frac{1}{\tau },\frac{1}{\tau},\frac{\sqrt{1-\tau }-1}{\tau };1\right)\\
  &&+2 \text{G}\left(\frac{1}{\tau},\frac{1}{\tau },\frac{1}{\tau },\frac{1}{\tau },\frac{1}{\tau },\frac{1}{\tau},\frac{1}{\tau },\frac{1}{\tau },\frac{1}{\tau },\frac{1}{\tau},\frac{\sqrt{1-\tau }-1}{\tau };1\right)\\
  &=&-\frac{677 \pi ^{10} \mylog{2}}{475200}-\frac{2339 \pi ^8 \mylog{2} ^3}{680400}-\frac{79\pi ^6 \mylog{2}^5}{28350}-\frac{2 \pi ^4 \mylog{2}^7}{1575}-\frac{4 \pi ^2 \mylog{2}^9}{8505}+\frac{16 \mylog{2}^{11}}{155925}\\
  &&-\frac{2339 \pi ^8 \zeta_{3}}{453600}-\frac{79 \pi ^6 \mylog{2}^2 \zeta_{3}}{1890}-\frac{1}{15} \pi ^4 \mylog{2}^4 \zeta_{3}-\frac{8}{135} \pi ^2 \mylog{2}^6 \zeta_{3}+\frac{8}{315} \mylog{2}^8\zeta_{3}\\
  &&-\frac{1}{5} \pi ^4 \mylog{2} \zeta_{3}^2-\frac{8}{9} \pi ^2 \mylog{2}^3 \zeta_{3}^2+\frac{16}{15} \mylog{2}^5 \zeta_{3}^2-\frac{4 \pi ^2 \zeta_{3}^3}{9}+\frac{16}{3} \mylog{2}^2 \zeta_{3}^3-\frac{79 \pi ^6 \zeta_{5}}{1260}\\
  &&-\frac{3}{5} \pi ^4 \mylog{2}^2 \zeta_{5}-\frac{4}{3} \pi ^2 \mylog{2}^4\zeta_{5}+\frac{16}{15} \mylog{2}^6 \zeta_{5}-8 \pi ^2 \mylog{2} \zeta_{3} \zeta_{5}+32 \mylog{2}^3 \zeta_{3} \zeta_{5}+24 \zeta_{3}^2 \zeta_{5}\\
  &&+72 \mylog{2} \zeta_{5}^2-\frac{9 \pi ^4 \zeta_{7}}{10}-12 \pi ^2 \mylog{2}^2 \zeta_{7}+24 \mylog{2}^4\zeta_{7}+144 \mylog{2} \zeta_{3} \zeta_{7}-\frac{170 \pi ^2 \zeta_{9}}{9}\\
  &&+\frac{680}{3} \mylog{2}^2 \zeta_{9}+372 \zeta_{11}.
\end{eqnarray*}

Note that all the identities listed in Section~\ref{sec:2} can also be computed using rewrite rules. But using these rewrite rules turns out to be much more efficient. 
We are now going to list several additional identities that could be computed with the help of this command:
\begin{eqnarray}
\sum _{n=1}^{\infty } \frac{\left(\frac{1}{3}\right)_n S_3(n)}{(n+1)!}&=&-\frac{5\pi^3}{32\sqrt{3}}+\frac{9}{16}\sqrt{3}\pi h_1-\frac{15\pi^2l_3}{32}-\frac{81h_1l_3}{16}-\frac{9}{32}\sqrt{3}\pi l_3^2\nonumber\\
  &&+\frac{27l_3^3}{32}+6\zeta_3,\\
 \sum _{i=1}^{\infty } \frac{\left(\frac{1}{2}\right)_n S_{1,1,1,1,1}(n)}{(n+1)!}&=&60 \zeta_{5},\\
 \sum _{i=1}^{\infty } \frac{\left(\frac{1}{3}\right)_n S_{1,1,1,1,1}(n)}{(n+1)!}&=&180 \zeta_{5}-\frac{\pi ^5}{\sqrt{3}},\\
 \sum _{n=1}^{\infty } \frac{\left(\frac{1}{2}\right)_n S_4(n)}{n (n+1)!}&=&-\frac{\pi ^4}{20}-\frac{2}{3} \pi ^2 l_2^2-\frac{4}{9} \pi ^2 l_2^3+\frac{4 l_2^4}{3}+\frac{8 l_2^5}{15}+16 l_2 p_4+16 p_5\nonumber\\
 &&-\frac{7 \pi ^2 \zeta _3}{12}+8 l_2 \zeta _3+7 l_2^2 \zeta _3-\frac{63 \zeta _5}{8},\\
 \sum _{n=1}^{\infty } \frac{\left(\frac{1}{2}\right)_n S_{3,1}(n)}{(n+1)! n^3}&=&\frac{13 \pi ^4}{180}+\frac{\pi ^6}{189}+\frac{199 \pi ^6 l_2}{7560}-\frac{2}{3} \pi ^2 l_2^2-\frac{4}{27} \pi ^4 l_2^3-\frac{2 l_2^4}{3}+\frac{8}{45} \pi ^2 l_2^5\nonumber\\
 &&-16 p_4+\frac{16}{3} \pi ^2 l_2 p_4+\frac{16 \pi ^2 p_5}{3}+24 s_1+\frac{200 l_2 s_1}{7}+\frac{136 s_2}{7}\nonumber\\
 &&-\frac{200 s_3}{7}-\frac{3 \pi ^2 \zeta _3}{2}+\frac{29 \pi ^4 \zeta _3}{168}+4 l_2 \zeta _3-3 \pi ^2 l_2 \zeta _3+\frac{5}{6} \pi ^2 l_2^2 \zeta _3\nonumber\\
 &&-\frac{3}{2} l_2^4 \zeta _3-36p_4 \zeta _3+\frac{9 \zeta _3^2}{8}-\frac{243}{7} l_2 \zeta _3^2+\frac{75 \zeta _5}{4}-\frac{2935 \pi ^2 \zeta _5}{168}\nonumber\\
 &&-9 l_2 \zeta _5-\frac{111}{2} l_2^2 \zeta _5+\frac{12685 \zeta _7}{112},\\
 \sum _{n=1}^{\infty } \frac{\left(\frac{1}{2}\right)_n S_3(n)}{(n+1)! n^5}&=&\frac{37\pi^4}{360}+\frac{89\pi^6}{5040}-\frac{63031\pi^8}{3024000}+\frac{2}{3}\pi^2\mylog{2}
 +\frac{37}{180}\pi^4\mylog{2}+\frac{463\pi^6\mylog{2}}{7560}\nonumber\\
 &&+\frac{1}{3}\pi^2\mylog{2}^2+\frac{47}{180}\pi^4\mylog{2}^2+\frac{1079\pi^6\mylog{2}^2}{7560}-\frac{8\mylog{2}^3}{3}+\frac{2}{9}\pi^2\mylog{2}^3+\frac{47}{270}\pi^4\mylog{2}^3-\frac{\mylog{2}^4}{3}\nonumber\\
 &&-\frac{5}{18}\pi^2\mylog{2}^4+\frac{19}{180}\pi^4\mylog{2}^4-\frac{2\mylog{2}^5}{15}-\frac{1}{9}\pi^2\mylog{2}^5+\frac{2\mylog{2}^6}{5}-\frac{1}{27}\pi^2\mylog{2}^6+\frac{4\mylog{2}^7}{35}+\frac{\mylog{2}^8}{35}\nonumber\\
 &&-8p_4-\frac{4}{3}\pi^2 p_4+\frac{4}{9}\pi^4p_4+16\mylog{2}^2p_4+16p_5+\frac{8}{3}\pi^2p_5+32\mylog{2}p_5\nonumber\\
 &&-32\mylog{2}^2p_5-\frac{16}{3}\pi^2p_6-128\mylog{2}p_6+64\mylog{2}^2p_6-128p_7+384\mylog{2}p_7\nonumber\\
 &&+1100s_5-\frac{134}{3}\pi^2s_1-32\mylog{2}s_1-104\mylog{2}^2s_1+\frac{6939}{40}s_4-16s_3 \nonumber\\
 &&+32s_2-64\mylog{2}s_2+128s_680s_7-4\zeta_{3}+\frac{7\pi^2\zeta_{3}}{12}-\frac{13\pi^4\zeta_{3}}{60}-7\mylog{2}\zeta_{3}\nonumber\\
 &&-\frac{271}{90}\pi^4\mylog{2}\zeta_{3}-7\mylog{2}^2\zeta_{3}-\frac{20}{9}\pi^2\mylog{2}^3\zeta_{3}+\frac{4}{3}\mylog{2}^5\zeta_{3}-160p_5\zeta_{3}-\frac{43\pi^2\zeta_{3}^2}{2}\nonumber\\
 &&-9\mylog{2}\zeta_{3}^2+32\mylog{2}^2\zeta_{3}^2-\frac{203\zeta_{5}}{8}-\frac{249\pi^2\zeta_{5}}{16}-\frac{203}{4}\mylog{2}\zeta_{5}-\frac{361}{12}\pi^2\mylog{2}\zeta_{5}\nonumber\\
 &&-\frac{203}{4}\mylog{2}^2\zeta_{5}+\frac{201}{2}\mylog{2}^3\zeta_{5}+\frac{393\zeta_{3}\zeta_{5}}{8}+\frac{3955\zeta_{7}}{16}-\frac{11533}{16}\mylog{2}\zeta_{7}\nonumber\\
 &&+640p_8+48\mylog{2}s_3.
\end{eqnarray}

To conclude this section we consider the sum $$\sum _{n=1}^{\infty } \frac{\left(-\frac{1}{2}\right)_n S_{(3,1,1)}(n)}{(n-1)! n}.$$ We find that it equals 
\begin{eqnarray}\label{notcyclo}
\frac{1-2 \text{G}\left(\frac{\sqrt{1-\tau }}{1-\tau ^{1/3}};1\right)+2 \text{G}\left(\frac{\sqrt{1-\tau }}{1+\tau ^{1/3}+\tau ^{2/3}};1\right)-33 \text{G}\left(\sqrt{1-\tau } \tau ^{1/3};1\right)-2
   \text{G}\left(\frac{\sqrt{1-\tau } \tau ^{1/3}}{1+\tau ^{1/3}+\tau ^{2/3}};1\right)}{45}.
\end{eqnarray}
Here we fail to transform the iterated integrals in terms of cyclotomic polylogarithms, however, since the integrals are simple enough, we are able to perform the integrals in~(\ref{notcyclo}) for example by using \texttt{Mathematica} and find the result
\begin{eqnarray*}
\sum _{n=1}^{\infty } \frac{\left(-\frac{1}{2}\right)_n S_{(3,1,1)}(n)}{(n-1)! n}=-\frac{4 \pi ^{3/2}}{27 \sqrt{3} \Gamma \left(\frac{5}{3}\right) \Gamma \left(\frac{11}{6}\right)}.
\end{eqnarray*}
Another example where we fail to transform the iterated integrals in terms of cyclotomic polylogarithms but still can do the integrals is:
\begin{eqnarray*}
\sum _{n=1}^{\infty }  \frac{\left(\frac{1}{3}\right)_n \left(S_{(3,1,1)}(n)-S_{(3,2,1)}(n)\right)}{(n+1)!}&=&-\frac{3}{4}+\frac{\pi }{\sqrt{3}}-\frac{\sqrt{3 \pi } \Gamma \left(\frac{5}{6}\right)}{\sqrt[3]{2} \Gamma \left(\frac{1}{3}\right)}.
\end{eqnarray*}
In the following section we will consider a subclass of Pochhammer sums, for which we will always be able to derive a representation in terms of cyclotomic polylogarithms.

\section{Using Rewrite Rules to directly derive Generating Functions in terms of Cyclotomic Polylogarithms}
\label{sec:4}
In this section we will deal with a sub class of the Pochhammer sums, namely we restrict the inner sum to be a multiple harmonic sum and we set $p=1/q$ with $q\in\Z\setminus\{0\}$ and $a=1$ in~(\ref{Pochhammersum}) \ie we are considering sums of the form 
\begin{eqnarray}\label{PochhammersumsH}
   \sum_{n=1}^\infty \frac{(p)_n}{(n+b)^c(n+d)!}\S{c_1,c_2,\ldots,c_r}n,
\end{eqnarray}
where $c,c_i\in\N,\ b,d\in\Z$ and $p=\frac{1}{q}$ with $q\in\Z\setminus\{0\}.$
Considering a Pochhammer sum in this subclass we could again use the rewrite rules presented in Section~\ref{sec:3} to find an integral representation, however we can also use the following lemmas. 
These new rewrite rules will directly lead to cyclotomic polylogarithms. 
We again start with the base cases where no inner sum is present (compare Lemma~\ref{GLBaseCase}):

\begin{lemma}\label{lemma:GenFunPochhammerSumToH1} Let $\K$ be a field of characteristic 0. Then the following identities hold in the ring $\K[[x]]$ of formal power series with $c\in\N$ and $b,d\in\Z$:
 \begin{eqnarray}
   \sum_{n=1}^\infty x^n\frac{(p)_n}{(n+d)!} &=& -\frac{(1-x)^{d-p}x^{-d}}{\abs{p}\, d!}\int_1^{(1-x)^{\abs{p}}} \frac{(1-t^{\frac{1}{\abs{p}}})^{d}}{t^{1-\sign{p}}\;t^{\frac{d}{\abs{p}}}} dt,\ d>0, \label{eq:GenFunPochhammerSumToH03}\\
   \sum_{n=1}^\infty x^n\frac{(p)_n}{(n+b)^c(n+d)!} &=& \frac{-1}{\abs{p}\, x^{b}}\int_1^{(1-x)^{\abs{p}}}\frac{\left(1-t^{\frac{1}{\abs{p}}}\right)^{b-1}}{t^{1-\frac{1}{\abs{p}}}}\sum_{n=1}^\infty \frac{(p)_n\left(1-t^{\frac{1}{\abs{p}}}\right)^n}{(n+b)^{c-1}(n+d)!}dt.\label{eq:GenFunPochhammerSumToH04}
 \end{eqnarray}
\end{lemma}
In the cases where there is an inner multiple harmonic sum present we can refine the Lemmas~\ref{GLc0dl0},~\ref{GLc0dgeq0} and~\ref{GLcgeq1} and get the following result.
\begin{lemma}\label{lemma:GenFunPochhammerSumToH2} Let $\K$ be a field of characteristic 0 and let $f:\N \to \K.$ Then the following identities hold in the ring $\K[[x]]$ of formal power series with $c\in\N,\ b,d\in\Z$ and $\S{m_1,\ldots,m_r}n$ a multiple harmonic sum:
 \begin{eqnarray}
   &&c=0,\ d<0:\nonumber\\
   &&\sum_{n=1}^\infty x^n\frac{(p)_n}{(n+d)!} \S{m_1,\ldots,m_r}n=\label{eq:GenFunPochhammerSumToH05}\\
   &&\frac{(1-x)^{d-p}}{x^{d}}\left((p)_{-d}s(-d)-\frac{1}{\abs{p}}\int_1^{(1-x)^{\abs{p}}} \frac{\left(1-t^{\frac{1}{\abs{p}}}\right)^{d-1}}{t^{1-\sign{p}}\;t^{\frac{d}{\abs{p}}}}\sum_{n=1}^\infty\frac{\left(1-t^{\frac{1}{\abs{p}}}\right)^n (p)_n}{(n+d-1)!n^{m_1}}\bar{s}(n) dt\right); \nonumber\\ 
   &&c=0,\ d\geq0:\nonumber\\
   &&\sum_{n=1}^\infty x^n\frac{(p)_n}{(n+d)!} \S{m_1,\ldots,m_r}n= \label{eq:GenFunPochhammerSumToH06}\\
   &&-\frac{(1-x)^{d-p}x^{-d}}{\abs{p}}\int_1^{(1-x)^{\abs{p}}} \frac{\left(1-t^{\frac{1}{\abs{p}}}\right)^{d-1}}{t^{1-\sign{p}}\;t^{\frac{d}{\abs{p}}}}\sum_{n=1}^\infty\frac{\left(1-t^{\frac{1}{\abs{p}}}\right)^n(p)_n}{(n+d-1)!n^{m_1}}\bar{s}(n) dt; \nonumber\\  
   &&c>0:\nonumber\\
   &&\sum_{n=1}^\infty x^n\frac{(p)_n}{(n+b)^c(n+d)!} \S{m_1,\ldots,m_r}n=\label{eq:GenFunPochhammerSumToH07}\\
   &&-\frac{x^{-b}}{\abs{p}}\int_1^{(1-x)^{\abs{p}}}t^{\frac{1}{\abs{p}}-1}\left(1-t^{\frac{1}{\abs{p}}}\right)^{b-1}\sum_{n=1}^\infty \left(1-t^{\frac{1}{\abs{p}}}\right)^n \frac{(p)_n}{(n+b)^{c-1}(n+d)!}s(n)dt.\nonumber
 \end{eqnarray}
 Here we use the abbreviations $s(n):=\S{m_1,\ldots,m_r}n$ and $\bar{s}(n):=\S{m_2,\ldots,m_r}n.$
\end{lemma}
\begin{proof}
 For all these equalities it is possible to find an initial value problem, which has a unique solution near $x=0$ and is satisfied by both sides of the respective equation.
\end{proof}
Note that the polynomials arising in the left hand sides of the equations in  Lemma~\ref{lemma:GenFunPochhammerSumToH1} and Lemma~\ref{lemma:GenFunPochhammerSumToH2} are of the form $t^i$ or $(1-t^i)^k$ for $i,k\in\Z,$ 
hence integrating over these integrands will lead to cyclotomic polylogarithms. Therefore the Pochhammer sums of the form~(\ref{PochhammersumsH}) will be expressible in terms of cyclotomic polylogarithms, and we can state the following structural theorem.
\begin{theorem}
 Any sum of the form 
 \begin{eqnarray}
   \sum_{n=1}^\infty \frac{\left(\frac{1}q\right)_n}{(n+b)^c(n+d)!}\S{c_1,c_2,\ldots,c_r}n,
\end{eqnarray}
where $c,c_i\in\N,\ b,d\in\Z$ and $q\in\Z\setminus\{0\},$ can be expressed in terms of cyclotomic polylogarithms at one.
\end{theorem}

Let us now, for the third time, consider~(\ref{GeneralExampleSum}) and illustrate how the previous lemmas can be used as rewrite rules to directly find a representation in terms of cyclotomic polylogarithms.
\begin{example}
We seek a closed form representation in terms of cyclotomic polylogarithms of 
$$\sum _{n=1}^{\infty } x^n \frac{\left(-\frac{1}{2}\right)_nS_1(n) }{(3+n)^2 (n-1)!},$$
so we use~(\ref{eq:GenFunPochhammerSumToH07}) twice:
\begin{eqnarray*}
  &&\sum _{n=1}^{\infty } x^n\frac{\left(-\frac{1}{2}\right)_nS_1(n) }{(3+n)^2 (n-1)!}=\\
  &&\hspace{2cm}=-\frac{2}{x^3} \int_1^{\sqrt{1-x}} t \left(1-t^2\right)^2 \sum _{n=1}^{\infty } \frac{\left(1-t^2\right)^n \left(-\frac{1}{2}\right)_n S_1(n)}{(3+n) (n-1)!} \, dt\\
  &&\hspace{2cm}=\frac{4}{x^3} \int_1^{\sqrt{1-x}} \frac{t}{1-t^2} \int_1^t u \left(1-u^2\right)^2 \sum _{n=1}^{\infty } \frac{\left(1-u^2\right)^n \left(-\frac{1}{2}\right)_n S_1(n)}{(n-1)!} \, du \, dt.
\end{eqnarray*}
Now we apply~(\ref{eq:GenFunPochhammerSumToH05}) followed by applying~(\ref{eq:GenFunPochhammerSumToH04}) and~(\ref{eq:GenFunPochhammerSum01}):
\begin{eqnarray*}
  &&\sum _{n=1}^{\infty } x^n\frac{\left(-\frac{1}{2}\right)_nS_1(n) }{(3+n)^2 (n-1)!}\\
  &&\hspace{1cm}=-\frac{4}{x^3} \int_1^{\sqrt{1-x}} \frac{t}{1-t^2} \int_1^t \left(1-u^2\right)^3 \left(2 \int_1^u \frac{\sum\limits_{n=1}^{\infty } \frac{\left(1-v^2\right)^n \left(-\frac{1}{2}\right)_n}{n (n-2)!}}{\left(1-v^2\right)^2} \,
   dv+\frac{1}{2}\right) \, du \, dt\\
  &&\hspace{1cm}=\frac{4}{x^3} \int_1^{\sqrt{1-x}} \frac{t}{1-t^2}  \int_1^t \left(1-u^2\right)^3 \left(\int_1^u \frac{\int_1^v \frac{4w \sum\limits _{n=1}^{\infty } \frac{\left(1-w^2\right)^n \left(-\frac{1}{2}\right)_n}{(n-2)!}}{1-w^2} \,
   dw}{\left(1-v^2\right)^2} \, dv-\frac{1}{2}\right) \, du\, dt\\
  &&\hspace{1cm}=\frac{4}{x^3} \int_1^{\sqrt{1-x}} \frac{t}{1-t^2} \int_1^t \left(1-u^2\right)^3 \left(\int_1^u \frac{\int_1^v \frac{w^2-1}{w^2} \, dw}{\left(1-v^2\right)^2} \, dv-\frac{1}{2}\right) \, du \, dt.
\end{eqnarray*}
Now we can send $x\to1$ and rewrite this expression directly in terms of cyclotomic harmonic polylogarithms (again this can be done by hand or by using the command \texttt{GLIntegrate} of \HarmonicSumsP) and arrive again at
\begin{eqnarray}\label{ExpampleHRep}
-\frac{9367}{7350}-\frac{64 H_{(2,0)}(1)}{3675}-\frac{32}{35} H_{(0,0),(1,0)}(1)-\frac{32}{35} H_{(2,0),(0,0)}(1)+\frac{64}{35} H_{(2,0),(1,0)}(1).
\end{eqnarray}
Finally, we can again use relations between cyclotomic polylogarithms at one to derive~(\ref{ExampleResult}).
\end{example}

Note that this is implemented in the command \texttt{PochhammerSumToH}, so calling
\begin{eqnarray*}
\textbf{PochhammerSumToH}\left[\frac{\left(-\frac{1}{2}\right)_n S_1(n)}{(3+n)^2 (-1+n)!},x,\{n,1,\infty \}\right]/.x\to1
\end{eqnarray*}
will immediately give~(\ref{ExpampleHRep}).

To conclude we are going to list several identities that could be computed with the help of this command (note that these identities could have also be computed using the methods presented in the previous sections):
\begin{eqnarray}
\sum _{n=1}^{\infty } \frac{\left(\frac{1}{5}\right)_n S_1(n)}{(n+1)!}&=&\frac{25 h_6}{4},\\
\sum _{n=1}^{\infty } \frac{\left(\frac{1}{5}\right)_n S_2(n)}{(n+1)!}&=&\frac{875 h_6^2}{48}+\frac{125}{12} \sqrt{5} h_6^2+\frac{125 h_7}{16}+\frac{125 h_8}{8},\\
\sum _{n=1}^{\infty } \frac{\left(\frac{1}{2}\right)_n S_{2,2,2}(n)}{(n+1)!}&=&\frac{2 \pi ^6}{189}-\frac{9 \zeta _3^2}{4}-\frac{15 l_2 \zeta _5}{2},\\
\sum _{n=1}^{\infty } \frac{\left(\frac{1}{2}\right)_n S_{2,2}(n)}{(n+2)!}&=&\frac{2 \pi ^2}{9}+\frac{\pi ^4}{45}-\frac{8 l_2^2}{3}-\zeta _3-2 l_2 \zeta _3,\\
\sum _{n=1}^{\infty } \frac{\left(\frac{1}{2}\right)_n S_{2,2}(n)}{n (n+2)!}&=&-\frac{\pi ^2}{9}-\frac{2 \pi ^4}{45}+\frac{4 l_2^2}{3}+\frac{\zeta _3}{2}+4 l_2 \zeta _3+\frac{15 \zeta _5}{16},\\
\sum _{n=1}^{\infty } \frac{\left(\frac{1}{3}\right)_n S_{2,2}(n)}{(n+3)!}&=&\frac{9}{320}+\frac{3 \sqrt{3} \pi }{320}+\frac{51 \pi ^2}{512}+\frac{11 \pi ^3}{128 \sqrt{3}}+\frac{91 \pi ^4}{5760}+\frac{1377 h_1}{1280}\nonumber\\
&&-\frac{27}{128} \sqrt{3} \pi  h_1-\frac{135 \pi ^2 h_1}{128}+\frac{243h_2}{64}+\frac{27}{64} \sqrt{3} \pi  h_2+\frac{1539 h_3}{128}\nonumber\\
&&-\frac{243 h_4}{16}-\frac{27 l_3}{320}+\frac{153 \sqrt{3} \pi  l_3}{1280}+\frac{11 \pi ^3 l_3}{128 \sqrt{3}}-\frac{27}{128} \sqrt{3} \pi  h_1 l_3\nonumber\\
&&+\frac{243 h_2l_3}{64}-\frac{1377 l_3^2}{2560}-\frac{81 \zeta _3}{64}+\frac{39}{64} \sqrt{3} \pi  \zeta _3-\frac{81 l_3 \zeta _3}{64}.
\end{eqnarray}

\subsection*{Acknowledgements}
The author would like to thank C. Schneider for useful discussions.


\begin{thebibliography}{10}

\bibitem{BinomialSumIdentities}
  J.~Ablinger.
  \newblock Discovering and Proving Infinite Binomial Sums Identities. \newblock
  {\em J. Exp. Math.},  26, 2017.
  \arxiv{1507.01703}

\bibitem{InvMellinKovacic}
  J.~Ablinger,
  \newblock Computing the Inverse Mellin Transform of Holonomic Sequences using Kovacic's Algorithm. \newblock 
  in~:  PoS RADCOR2017, 069, 2017.
  \arxiv{1801.01039}  
  
\bibitem{InvMellin}
  J.~Ablinger,
  \newblock Inverse Mellin Transform of Holonomic Sequences. \newblock
  {\em PoS LL {\bf 2016}}, 067, 2016.
  \arxiv{1606.02845}
  
\bibitem{HarmonicSums}
  J.~Ablinger.
  \newblock The package HarmonicSums: Computer Algebra and Analytic aspects of Nested Sums. \newblock
  in~:  Loops and Legs in Quantum Field Theory - LL 2014.
  \arxiv{1407.6180}

\bibitem{Ablinger:2013jta}
  J.~Ablinger and J.~Bl\"umlein.
  \newblock Harmonic Sums, Polylogarithms, Special Numbers, and their Generalizations. \newblock
  in~:  Computer Algebra in Quantum Field Theory: Integration, Summation and Special Functions, \newblock 
  Texts \& Monographs in Symbolic Computation, Eds. C.~Schneider and J.~Bl\"umlein, (Springer, Wien, 2013), pp.~1--32. 
  \arxiv{1304.7071}
  
\bibitem{Ablinger:2013cf}
  J.~Ablinger, J.~Bl\"umlein and C.~Schneider.
  \newblock Analytic and Algorithmic Aspects of Generalized Harmonic Sums and Polylogarithms. \newblock
  {\em J. Math. Phys.}  54, 2013.
  \arxiv{1302.0378}  

\bibitem{Ablinger:2013eba}
  J. Ablinger, J. Bl\"umlein and C. Schneider.
  \newblock Generalized Harmonic, Cyclotomic, and Binomial Sums, their Polylogarithms and Special Numbers. \newblock
  {\em J. Phys. Conf. Ser.},  523, 2014.
  \arxiv{1310.5645}

\bibitem{Ablinger:2013hcp}
  J.~Ablinger.
  \newblock Computer Algebra Algorithms for Special Functions in Particle Physics. \newblock
  \arxiv{1305.0687}

\bibitem{Ablinger:2014}
  J.~Ablinger and J.~Bl\"umlein and C.G.~Raab and C.~Schneider.
  \newblock Iterated Binomial Sums and their Associated Iterated Integrals. \newblock
  {\em J. Math. Phys. Comput}, 55:1--57, 2014.
  \arxiv{1407.1822}

\bibitem{Ablinger:2011te}
  J.~Ablinger, J.~Bl\"umlein and C.~Schneider.
  \newblock Harmonic Sums and Polylogarithms Generated by Cyclotomic Polynomials. \newblock
  {\em J. Math. Phys.},  52, 2011.
  \arxiv{1105.6063}

\bibitem{Abramov:94}
  S.A.~Abramov and M.~Petkov{\v s}ek. 
  \newblock D'{A}lembertian solutions of linear differential and difference equations. \newblock 
  in proceedings of {\em ISSAC'94}, ACM Press, 1994.

\bibitem{Bluemlein1999}
  J. Bl\"umlein and S.~Kurth.
  \newblock Harmonic sums and Mellin transforms up to two-loop order. \newblock
  {\em Phys. Rev.},  D60 014018, 1999.
  \arxiv{hep-ph/9810241v2}  
  
\bibitem{Bluemlein2000}
  J. Bl\"umlein.
  \newblock Analytic continuation of Mellin transforms up to two-loop order. \newblock
  {\em Comput. Phys. Commun.},  133, 2000.
  \arxiv{0003100}  
  
\bibitem{Borwein:1987}
  J.M.~Borwein and P.B.~Borwein, {\em Pi and the AGM: A Study in Analytic Number Theory and Computational Complexity}, Wiley, New York, 1987. Reprinted 1998.

\bibitem{Borwein:2001}
  J.M.~Borwein and D.J.~Broadhurst and J.~Kamnitzer.
  \newblock Central binomial sums, multiple {C}lausen values, and zeta values. \newblock
  {\em  Experiment. Math.} 10:25--34, 2001.
  \arxiv{hep-th/0004153}

\bibitem{Borwein:2000}
  J.M.~Borwein and P. Lison\v{e}k.
 \newblock Applications of integer relation algorithms. \newblock
  {\em Discrete Math.}, 217:65--82, 2000.  

\bibitem{Bronstein}
    M.~Bronstein.
    {\em Linear Ordinary Differential Equations: breaking through the order 2 barrier},
    in proceedings of \emph{ISSAC'92}, ACM Press, 1992.

\bibitem{Davydychev:2001}
  A.~I.~Davydychev and M.~Y.~Kalmykov.
  \newblock New results for the epsilon-expansion of certain one-, two- and three-loop Feynman diagrams.
  {\em Nucl. Phys. B}, 605:266--318, 2001.
  \arxiv{hep-th/0012189}    
 
\bibitem{Davydychev:2003mv}
  A.~I.~Davydychev and M.~Y.~Kalmykov.
  \newblock Massive Feynman diagrams and inverse binomial sums.
  {\em Nucl. Phys. B}, 699:3--64, 2004.
  \arxiv{hep-th/0303162}

\bibitem{Fleischer:1998nb}
  J.~Fleischer, A.V.~Kotikov and O.L.~Veretin.
  \newblock Analytic two loop results for selfenergy type and vertex type diagrams with one nonzero 
  mass. \newblock
  {\em Nucl. Phys. B}, 547:343--374, 1999.
  \arxiv{hep-ph/9808242}  

\bibitem{Jegerlehner:2003}
  F.~Jegerlehner, M.~Y.~Kalmykov and O.~Veretin.
  \newblock MS Versus Pole Masses of Gauge Bosons II: Two-Loop Electroweak Fermion Corrections. \newblock
  {\em Nucl. Phys. B}, 658:49-112, 2003.
  \arxiv{hep-ph/0212319}  
  
\bibitem{Kalmykov:2000qe}
  M.~Y.~Kalmykov and O.~Veretin.
  \newblock Single scale diagrams and multiple binomial sums. \newblock
  {\em Phys.\ Lett.\ B}, 483:315--323, 2000.
  \arxiv{hep-th/0004010}
  
\bibitem{Kalmykov:2007dk}
  M.Y.~Kalmykov, B.F.L.~Ward and S.A.~Yost.
  \newblock Multiple (inverse) binomial sums of arbitrary weight and depth and the all-order 
  $\varepsilon$-expansion of generalized hypergeometric functions with one half-integer value of parameter. \newblock
  JHEP {\bf 0710} (2007) 048,
  \arxiv{0707.3654}  
 
\bibitem{KauersPaule:2011}
  M.~Kauers and P.~Paule. 
  \emph{The Concrete Tetrahedron}, Text and Monographs in Symbolic Computation, Springer, Wien, 2011.  
  
\bibitem{Kovacic}  
  J.J.~Kovacic.
  \newblock An algorithm for solving second order linear homogeneous differential equations. \newblock 
  {\em J.~Symbolic Comput.}, 2, 1986.  
  
\bibitem{Lehmer:1985}
  D. H.~Lehmer.
  \newblock Interesting series involving the central binomial coefficient. \newblock
  {\em Amer. Math. Monthly}, 92:449--457, 1985.

\bibitem{Liu:2019}
  H. Liu and W. Wang.
  \newblock Gauss's theorem and harmonic number summation formulae with certain mathematical constants. \newblock
  {\em J. of of Difference Equations and Applications}, 1--18, 2019.  
  \doi{10.1080/10236198.2019.1572127}   
  
\bibitem{Ogreid:1997bx}
  O.~M.~Ogreid and P.~Osland.
  \newblock Summing one-dimensional and two-dimensional series related to the Euler series. \newblock
  {\em J. Comput. Appl. Math.},  98:245--271, 1998.
  \arxiv{hep-th/9801168} 

\bibitem{Petkov:92}
  M.~Petkov{\v s}ek, 
  \newblock Hypergeometric solutions of linear recurrences with polynomial coefficients. \newblock
  {\em J.~Symbolic Comput.}, 14, 1992.  

\bibitem{Remiddi:1999ew}
  E.~Remiddi and J.A.M.~Vermaseren.
  \newblock Harmonic polylogarithms. \newblock
  {\em Int. J. Mod. Phys. A}, 15:725--754, 2000.
  \arxiv{hep-ph/9905237}

\bibitem{Singer:99}
  P.A. Hendriks and M.F. Singer.
  \newblock Solving difference equations in finite terms. \newblock
  {\em J.~Symbolic Comput.}, 27, 1999.  
  
\bibitem{Vermaseren1998}
  J.A.M.~Vermaseren.
  \newblock Harmonic sums, {M}ellin transforms and {I}ntegrals. \newblock
  {\em Int. J. Mod. Phys.}, A14, 1999.
  \arxiv{9806280v1}

\bibitem{Weinzierl:2004bn}
  S.~Weinzierl.
  \newblock  Expansion around half integer values, binomial sums and inverse binomial sums.\newblock
  {\em J. Math. Phys.}, 45:2656--2673, 2004.
  \arxiv{hep-ph/0402131}

\bibitem{ZhiWei:2014}
  Zhi-Wei Sun.
  \newblock List of conjectural series for powers of $\pi$ and other constants. \newblock
  \arxiv{1102.5649}
  
\bibitem{Zucker:1985}
  I. J.~Zucker.
  \newblock On the series $\sum_{k=1}^\infty \binom{2k}{k}^{-1}k^{-n}$ and related sums. \newblock
  {\em J. Number Theory}, 20:92--102, 1985.
 

  
\end{thebibliography}
\end{document}